\documentclass[11pt]{amsart}
\usepackage{amssymb,latexsym}
\usepackage{xcolor}
\usepackage{float}
\usepackage{ulem}
\usepackage[all]{xy}
\usepackage{enumerate}
\usepackage{url}
\usepackage{tikz}
\usetikzlibrary{calc}
\usetikzlibrary{intersections,backgrounds}
\newtheorem{thm}{Theorem}[section]
\newtheorem{cor}[thm]{Corollary}
\newtheorem{prop}[thm]{Proposition}
\newtheorem{lem}[thm]{Lemma}

\theoremstyle{definition}

\newtheorem{rem}[thm]{Remark}

\numberwithin{equation}{section}

\tikzset{every picture/.style={line width=1pt}}
\usetikzlibrary{decorations.markings}
\tikzset{->-/.style 2 args={
    postaction={decorate},
    decoration={markings, mark=at position #1 with {\arrow[thick, #2]{>}}}
    },
    ->-/.default={0.5}{}
}
\newcommand{\pc}[1]{
\draw (#1) --++(0,0.5) --++(1,1) -- ++(0,0.5);
\draw ($(#1)+(1,0)$) --++(0,0.5) --++(-0.3,0.3);
\draw ($(#1)+(0,2)$) --++(0,-0.5) --++(0.3,-0.3);
}
\newcommand{\nc}[1]{
\draw (#1) --++(0,0.5) --++(0.3,0.3);
\draw ($(#1)+(1,2)$) --++(0,-0.5) --++(-0.3,-0.3);
\draw ($(#1)+(1,0)$) --++(0,0.5) --++(-1,1) --+(0,0.5);
}
\newcommand{\npc}{
\begin{tikzpicture}[baseline=0.4pt, x=3mm, y=3mm]
\draw[thick] (0,0) -- (1,1);
\draw[thick] (1,0) -- (0.7, 0.3); 
\draw[thick] (0.35, 0.65) -- (0,1);
\end{tikzpicture}
}

\newcommand{\nnc}{
\begin{tikzpicture}[baseline=0.4pt, x=3mm, y=3mm]
\draw[thick] (1,0) -- (0,1);
\draw[thick] (0,0) -- (0.3, 0.3); 
\draw[thick] (0.65, 0.65) -- (1,1);
\end{tikzpicture}
}
\begin{document}
\title[Crossing matrix and extended first Johnson homomorphism]{The crossing matrix and the extended first Johnson homomorphism of a braid group}
\author[Y. Kuno]{Yusuke Kuno}
\address{Department of Mathematics, Tsuda University, 2-1-1 Tsuda-machi, Kodaira-shi, Tokyo 187-8577, Japan}
\email{kunotti@tsuda.ac.jp}
\author[Y. Yaguchi]{Yoshiro Yaguchi}
\address{Maebashi Institute of Technology, 460-1, Kamisadori-machi, Maebashi-shi, Gunma 371-0816, Japan}
\email{y.yaguchi@maebashi-it.ac.jp}

\date{}

\begin{abstract}
We compare two crossed homomorphisms on a braid group, one defined diagrammatically and the other defined algebraically.
We show that these crossed homomorphisms are essentially the same, and compute them in detail for simple braids, namely elements conjugate to the standard generators of the braid group or to their inverses. 
\end{abstract}

\subjclass[2020]{Primary 20F36; Secondary 20F34}

\keywords{braid group, simple braid, crossing matrix, Johnson homomorphism}
\maketitle

\section{Introduction} \label{sec:intro}

Let $m$ be a positive integer with $m\geq 2$ and $B_m$ the braid group of degree $m$.
In this paper, we study two crossed homomorphisms on $B_m$.

One crossed homomorphism is defined diagrammatically and called the {\it crossing matrix}, which was introduced by Burillo, Gutierrez, Krsti\'{c} and Nitecki~\cite{BGKN02}.
The crossing matrix $C(\beta)$ of an $m$-braid $\beta$ is an $m\times m$ matrix whose $(i,j)$-entry counts the crossings of the $i$-th strand over the $j$-th strand, taking the signs into account.
This yields a crossed homomorphism 
\[
C: B_m \to {\rm Mat}_m^0,
\]
where the target is the abelian group of $m\times m$ integer matrices with zero diagonal entries.

Burillo et al.\ \cite{BGKN02} completely characterized the crossing matrices of braids, thereby determining the image $C(B_m)$. 
From various motivations, the problem of determining $C(S)$ for certain subsets $S$ of $B_m$ has been studied.
For example, when $S$ is the set of positive pure braids, a conjecture completely characterizing $C(S)$ was proposed in \cite{BGKN02}. Several authors, including \cite{GN18, aya-yoshi, yuko-aya-yoshi}, have investigated this conjecture and obtained partial results, but a complete solution has not yet been achieved.

The other crossed homomorphism is defined from a more algebraic viewpoint.
Recall that the group $B_m$ acts naturally on the free group $F_m$ of rank $m$, and we have 
an embedding $B_m \hookrightarrow {\rm Aut}(F_m)$ called the Artin representation.
Let $H= F_m^{\rm abel}$ be the abelianization of $F_m$, which is free abelian of rank $m$.
In a cohomological study of ${\rm Aut}(F_m)$, Kawazumi~\cite{kawazumi} introduced a crossed homomorphism
$
\tau_1^{\theta}: {\rm Aut}(F_m) \to {\rm Hom}(H, H^{\otimes 2})
$
called the {\it extended first Johnson homomorphism} (associated with the standard Magnus expansion). 
By restriction, we obtain a map from $B_m$ which actually takes values in a smaller space:
\[
\tau_1^{\theta}: B_m \to {\rm Hom}(H, \wedge^2 H).
\]
In \cite{kawazumi06}, Kawazumi used this map to study the cohomology of the braid group with twisted coefficients.

Since both maps $C$ and $\tau_1^{\theta}$ are crossed homomorphisms, it is natural to seek their relationship.
Our first main result, which is Theorem~\ref{thm:tau=C2}, shows that they carry essentially the same information: there is an injective map $\delta: {\rm Mat}_m^0 \to {\rm Hom}(H, \wedge^2 H)$ such that $\tau_1^{\theta} = \delta \circ C$. 

Our second main result concerns the computation of these invariants for simple braids, that is, elements conjugate to a standard generator of $B_m$ or to its inverse. 
Each simple braid is determined by its sign (positive/negative) and the isotopy class of its cord, a certain arc in an $m$-marked disk. 
In Theorem~\ref{prop:f_i_for_simple}, we give a formula for the crossing matrix of a simple braid in terms of a homological invariant of the corresponding cord. This implies that we can determine the crossing matrices of simple braids. 

The {\it Hurwitz action} is a natural action of a braid group $B_m$ on the $m$-fold direct product $S^m$ of a {\it quandle} $S$, that is, a set with a binary operation 
which axiomatizes the notion of conjugation in a group (cf. \cite{Dehornoy}). The {\it Hurwitz equivalence problem} asks whether given two elements of $S^m$ are sent to each other by the Hurwitz action, or not. When $S$ is the set of simple braids of degree $n$, solving the problem implies a complete invariant of {\it surface braids of degree} $n$ {\it with} $m$ {\it branch points}, which are, roughly speaking, a $2$-dimensional version of classical braids of degree $n$ (cf. \cite{kamada}). However, no algorithm to solve the problem is known for general quandles, including the set of simple braids. We expect to obtain invariants of surface braids using Theorem~\ref{prop:f_i_for_simple}.

The map $\tau_1^{\theta}$ extends to a family of maps $\{ \tau_k^{\theta}\}_{k\ge 1}$, called the extended higher Johnson homomorphisms \cite{kawazumi}.
It would be interesting to extend the equality $\tau_1^{\theta} = \delta \circ C$ to a diagrammatic description of the higher invariants $\tau_2^{\theta}$, $\tau_3^{\theta},\ldots$ on the braid group.

This paper is organized as follows. 
In Section~\ref{sec:braids}, we introduce terminology for braids and recall the definition of the crossing matrix.
In Section~\ref{sec:Johnson}, we recall the definition of the extended first Johnson homomorphism (to be precise, associated with the standard Magnus expansion).
In Section~\ref{sec:Candtau}, we prove Theorem~\ref{thm:tau=C2}.
In Section~\ref{sec:Simple_braids}, we recall the definition of a cord of a simple braid, introduce its homological invariant, and prove Theorem~\ref{prop:f_i_for_simple}.
 
\subsubsection*{Convention}
Group actions are always from the left. 
All the homology groups that we consider will be over the integers. 
Let $G$ be a group acting on a commutative group $A$. 
Then, the {\it semi-direct product} $A \rtimes G$ is a group structure on the cartesian product $A \times G$ defined by
\[
(a_1, g_1)(a_2, g_2) = (a_1 + g_1\cdot a_2, g_1 g_2).
\]
A map $c: G \to A$ is called a {\it crossed homomorphism} if it satisfies
\[
c(g_1 g_2) = c(g_1) + g_1 \cdot c(g_2)
\]
for any $g_1, g_2 \in G$.

In this paper, we often consider the following situation.
Let $\varpi: G \to K$ be a group homomorphism, and suppose that $K$ acts on a commutative group $A$.
The group $G$ acts on $A$ through $\varpi$. Let $c: G \to A$ be a crossed homomorphism.
Then, we have the following group homomorphism:
\[
\widetilde{c}: G \to A \rtimes K, \quad
\widetilde{c}(g) = (c(g), \varpi(g)).
\]

\section{The crossing matrix of a braid group} \label{sec:braids}
Let $m$ be a fixed integer with $m\geq 2$.
Let $I = [0,1]$ be the unit interval in $\mathbb{R}$ and let $D = I \times I$ be the closed domain in $\mathbb{R}^2$. For $i\in \{1,\dots ,m\}$, let $q_i=(i/(m+1), 1/2)\in D$, and let $Q_m=\{q_1,\dots ,q_m\}$. A {\it geometric} $m$-{\it braid} $\beta$ is a 1-manifold embedded in $D\times I$ with the boundary condition $\partial \beta=\beta\cap \partial(D\times I)=Q_m\times \{0,1\}$ such that the restriction $\pi|_{\beta}:\beta\to I$ of the natural projection $\pi:D\times I\to I$ is an $m$-fold covering map. 
We call $Q_m\times \{0\}$ (resp. $Q_m\times \{1\}$) the set of {\it starting points} (resp. the set of {\it terminal points}) of geometric $m$-braids. 
Take a {\it string} $\alpha$ {\it of} a geometric $m$-braid $\beta$, that is a connected component of $\beta$. 
If $\alpha\cap (D\times \{0\})=\{q_i\}\times \{0\}$ (resp. $\alpha\cap (D\times \{1\})=\{q_j\}\times \{1\}$), then we call the point  $q_i$ (resp. the point $q_j$) the {\it starting point} (resp. the {\it terminal point}) {\it of} $\alpha$. 
The string of $\beta$ starting from $q_i$ is called the $i$-{\it th string of} $\beta$. 
When we draw pictures of geometric $m$-braids, we place their starting points on the top and terminal points on the bottom. 

Let $B_m$ be the set of the equivalence classes of geometric $m$-braids under ambient isotopies of $D \times I$ that fix $\partial (D \times I)$ pointwise. 
The set $B_m$ forms a group with the ``stacking product'' as follows: If $\beta_1$ and $\beta_2$ are (the equivalence classes of) $m$-braids then $\beta_1\beta_2$ is the (equivalence class of) $m$-braid obtained by placing $\beta_1$ above $\beta_2$ so that the terminal points of $\beta_1$ coincide with the starting points of $\beta_2$. We call the group $B_m$ the {\it braid group of degree} $m$. For simplicity, we denote the equivalence class of a geometric $m$-braid $\beta$ by the same letter $\beta\in B_m$, and the strings of an element of $B_m$ refer to those of its geometric representative.  

A geometric $m$-braid $\beta$ is called {\it generic} if its image under the projection $D \times I \to I \times I, (x,y,t) \mapsto (x,t)$ has transverse double points only.
Endowing the image of $\beta$ with the over/under information of strings around each transverse double point by a small break in the under arc, we obtain the {\it diagram} of a generic geometric $m$-braid $\beta$, which we denote by $\Gamma_{\beta}$.
The double points in $\Gamma_{\beta}$ are called {\it crossings}. 
There are two types of crossings: positive $\npc$ and negative $\nnc$.

Any element in $B_m$ admits a generic representative.
Abusing notation, for $\beta \in B_m$ we denote by $\Gamma_{\beta}$ the diagram of its generic representative (which is not unique).

It is well-known that $B_m$ has the following presentation \cite{artin}:
\begin{equation*}
B_m = 
 \left<\sigma_1, \dots, 
\sigma_{m-1}\ \left|\ \begin{array}{rrr}\sigma_i\sigma_j\sigma_i=\sigma_j\sigma_i\sigma_j
\         (|i-j|=1),\\
\sigma_i\sigma_j=\sigma_j\sigma_i \ (|i-j|>1)
                  \end{array}\right.\right>,
\end{equation*}
where $\sigma_i$ is called the $i$-th standard generator of $B_m$, and is defined as the $m$-braid shown in Figure~\ref{fig:sigma_i}. 

\begin{figure}
\[
\begin{tikzpicture}[x=6mm, y=6mm]
\draw (0,0) -- (0,3);
\draw (3,0) -- (3,3);
\draw (5,0) -- (7,3);
\draw (7,0) -- (6.2,1.2);
\draw (5,3) -- (5.8,1.8);
\draw (9,0) -- (9,3);
\draw (12,0) -- (12,3);
\draw (1.5,1.5) node{$\dots$};
\draw (10.5,1.5) node{$\dots$};
\draw (0,3.5) node{\small $1$};
\draw (3,3.5) node{\small $i-1$};
\draw (5,3.5) node{\small $i$};
\draw (7,3.5) node{\small $i+1$};
\draw (9,3.5) node{\small $i+2$};
\draw (12,3.5) node{$m$};
\end{tikzpicture}
\]
\caption{the standard generator $\sigma_i \in B_m$}
\label{fig:sigma_i}
\end{figure}
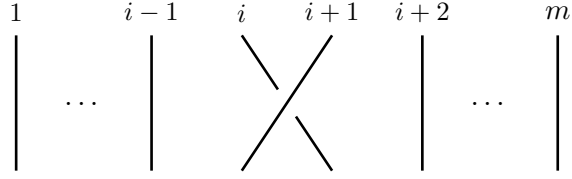

Let $S_m$ be the symmetric group of degree $m$, namely the set of bijections from $\{1,\dots ,m\}$ to itself.
The group operation of $S_m$ is the composition of functions, i.e. the product $\pi_1\pi_2$ is defined by $(\pi_1\pi_2)(j)=\pi_1(\pi_2(j))$ for each $\pi_1, \pi_2\in S_m$ and each $j\in \{1,\dots ,m\}$. 

Take an element $\beta\in B_m$. For each $j\in \{1,\dots ,m\}$, let $|\beta|(j)$ be the  
integer with $1\leq |\beta|(j)\leq m$ such that the terminal point of the $|\beta|(j)$-th string of $\beta$ is $q_j$. Then, we have a permutation $|\beta|\in S_m$ and have a surjective group homomorphism $|\cdot|: B_m\to S_m,\ \beta \mapsto |\beta|$. 

We recall the definition of the {\it crossing matrix} of $m$-braids \cite{BGKN02}.
Let $\beta \in B_m$ and take its diagram $\Gamma_{\beta}$.
The crossing matrix $C(\beta)$ of $\beta$ is defined to be the $m\times m$ matrix whose $(i,j)$ entry is the algebraic number of crossings in $\Gamma_{\beta}$ between the $i$-th and $j$-th strings, where the $i$-th string is over the $j$-th string.
Here, we assign $+1$ to positive crossings and $-1$ to negative crossings.
By definition, the diagonal entries of $C(\beta)$ are zero. 
Note that $C(\beta)$ does not depend on the choice of a diagram $\Gamma_{\beta}$. 
Figure~\ref{fig:beta_sample} shows an example for $\beta=\sigma_2^{-1}\sigma_1^2\sigma_2^3 \sigma_1^{-1}\sigma_2\in B_3$. 
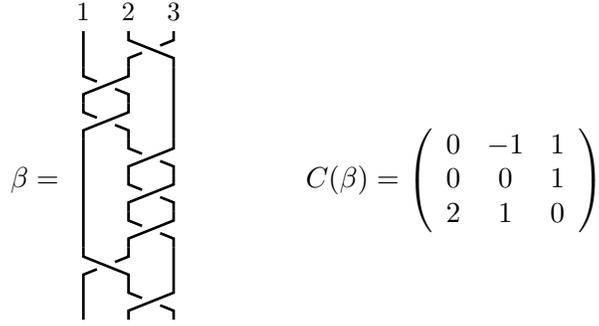
\begin{figure} 
\[
\beta= \begin{tikzpicture}[baseline=50pt, x=6mm, y=2.4mm]
\draw (0,16) node[above]{\small $1$};
\draw (1,16) node[above]{\small $2$};
\draw (2,16) node[above]{\small $3$};
\draw (0,0) -- (0,2);
\pc{1,0}
\nc{0,2}
\draw (2,2) -- (2,4);
\draw (0,4) -- (0,10);
\pc{1,4}
\pc{1,6}
\pc{1,8}
\pc{0,10}
\draw (2,10) -- (2,14);
\pc{0,12}
\draw (0,14) -- (0,16);
\nc{1,14}
\end{tikzpicture}
\hspace{4em} 
C(\beta)=\left (
\begin{array}{ccc}
0&-1&1\\
0&0&1\\
2&1&0
\end{array}
\right )
\]
\caption{the crossing matrix for $\beta = \sigma_2^{-1} \sigma_1^2 \sigma_2^3 \sigma_1^{-1} \sigma_2$}
\label{fig:beta_sample}
\end{figure}
Let ${\rm Mat}_m^0$ be the set of $m\times m$ integer matrices whose diagonal entries are all $0$. We have a map 
\[
C:B_m\to {\rm Mat}_m^0, \quad
\beta \mapsto C(\beta). 
\]

Matrix addition endows ${\rm Mat}_m^0$ with a commutative group structure. 
For $M \in {\rm Mat}_m^0$ and a permutation $\pi\in S_m$, let $\pi(M)\in {\rm Mat}_m^0$ be the matrix obtained by rearranging the rows and columns of $M$ according to $\pi$, i.e. the $(i,j)$ entry of $\pi(M)$ is the $(\pi^{-1}(i), \pi^{-1}(j))$ entry of $M$.
Then, we have an action of $S_m$ on ${\rm Mat}_m^0$.

Note that for any $\beta_1, \beta_2\in B_m$, the $i$-th string of $\beta_1$ is connected to the $|\beta_1|^{-1}(i)$-th string of $\beta_2$, and thus we have the following lemma.

\begin{lem}[\cite{BGKN02}] \label{lem:C_crossed}
The map $C$ is a crossed homomorphism. Namely, 
for any $\beta_1, \beta_2\in B_m$, $C(\beta_1\beta_2)=C(\beta_1)+|\beta_1|({C(\beta_2)})$.
Thus we have a group homomorphism $\widetilde{C}:B_m\to {\rm Mat}_m^0\rtimes S_m$ defined by $\widetilde{C}(\beta)=(C(\beta), |\beta|)$. 
\end{lem}

Let $H=H_m$ be the $m$-fold direct sum of the infinite cyclic group 
${\mathbb Z}$, which is generated by $m$ letters $X_1, \dots, X_m$.
The group $S_m$ acts on $H$ by $\pi(X_k) = X_{\pi(k)}$. 
For $M=(\alpha_{kl})\in {\rm Mat}_m^0$ and $i\in \{1,\dots, m\}$, let $f_i(M)$ be the element of $H$ defined by 
\[
f_i(M)=\sum_{k=1}^m \alpha_{ki} X_k. 
\]

\begin{lem} \label{lem:f_i_pi}
For any $\pi\in S_m$, any $M\in {\rm Mat}_m^0$ and any $i\in \{1,\dots ,m\}$, we have $f_i(\pi(M))=\pi(f_{\pi^{-1}(i)}(M))$.
\end{lem}

\begin{proof}
Put $(\alpha_{kl}')=\pi(M) \in {\rm Mat}_m^0$. Since $\alpha_{kl}'=\alpha_{\pi^{-1}(k)\pi^{-1}(l)}$, we have $\displaystyle f_i(\pi(M))=\sum_{k=1}^m \alpha_{ki}' X_k=\sum_{k=1}^m \alpha_{\pi^{-1}(k)\pi^{-1}(i)} X_k=\pi \left(\sum_{k=1}^m \alpha_{\pi^{-1}(k)\pi^{-1}(i)} X_{\pi^{-1}(k)}\right)\\=\pi \left(\sum_{k=1}^m \alpha_{k\pi^{-1}(i)} X_k\right)=\pi(f_{\pi^{-1}(i)}(M))$. Thus, we obtain the result.
\end{proof} 

Take $\beta\in B_m$ and $i\in \{1,\dots ,m\}$. Let $\displaystyle f_i(\beta)=f_i(C(\beta)) \in H$. Note that $f_i(\beta)$ tells us the algebraic count of how many times the $i$-th string of $\beta$ has passed under each of the other strings. 
We call $f_i(\beta)$ the $i$-{\it th diving combinational information of} $\beta$. Then, we have the following corollary.

\begin{lem} \label{lem:f_i_crossed}
For any $\beta_1, \beta_2\in B_m$ and any $i\in \{1,\dots ,m\}$, we have $f_i(\beta_1\beta_2)=f_i(\beta_1)+|\beta_1|f_{|\beta_1|^{-1}(i)}(\beta_2)$.
\end{lem}
\begin{proof}By Lemma~\ref{lem:C_crossed}, $f_i(C(\beta_1\beta_2))=f_i(C(\beta_1))+f_i(|\beta_1|C(\beta_2))$. By Lemma~\ref{lem:f_i_pi}, $f_i(|\beta_1|C(\beta_2))=|\beta_1|f_{|\beta_1|^{-1}(i)}(C(\beta_2))$, and we obtain the result.
\end{proof}

\begin{rem}\label{rem:the range of $C$}
The problem of determining the image of a subset of $B_m$ by the map $C: B_m\to {\rm Mat}_m^0$ has been studied as follows:

(1) Let $P_m$ be the {\it pure braid group of degree} $m$, namely the kernel of the homomorphism $|\cdot|: B_m\to S_m$. 
The restriction $C|_{P_m}$ is a group homomorphism by Lemma~\ref{lem:C_crossed}. It is known that $C(P_m)$ is the set of symmetric matrices that belong to ${\rm Mat}_m^0$ and ${\rm Ker}(C|_{P_m})$ is the commutator subgroup $[P_m, P_m]$ of $P_m$ (see \cite[Section 2]{GN18}).

(2)
An element of $B_m$ is called a {\it permutation braid} if it has a diagram whose each crossing is positive and each pair of strands has at most one crossing. Let $\Sigma_m$ be the set of permutation braids in $B_m$. The restriction of the map $|\cdot|:B_m\to S_m$ to $\Sigma_m$ is bijective.
As explained in \cite[p. 296]{BGKN02}, using \cite[Lemma~9.1.6]{thurston} directly, we have that the map $C|_{\Sigma_m} 
: \Sigma_m \to {\rm Mat}_m^0$ is injective and its image $C(\Sigma_m)$ is the set of matrices $L=(L_{ij})\in {\rm Mat}_m^0$ satisfying the following conditions: (i) $L_{ij}=0$ if $i\geq j$;
 (ii) $L_{ij}\in \{0,1\}$; and
 (iii) $L_{ij} =
 L_{jk} = p$ implies $L_{ik} = p$ for all $i,j,k$ with $1\le i < j < k \le m$ and $p \in \{ 0,1\}$.

(3) 
For $\pi\in S_m$, let $\pi^{+}\in B_m$ be the permutation braid with $|\cdot|(\pi^{+})=\pi$. For $\beta\in B_m$, let $a=\beta{(|\beta|^{+})}^{-1}\in P_m$. By Lemma~\ref{lem:C_crossed}, $C(\beta)=C(a)+C({|\beta|}^{+})$. By (1) and (2), 
we obtain that $C(B_m)$ is precisely the set of all sums $M+L$, where $M\in {\rm Mat}_m^0$ is symmetric and $L\in {\rm Mat}_m^0$ satisfies the conditions (i)–(iii) in (2); see also \cite[Corollary 3.2]{BGKN02}.

(4) 
A much harder problem is to determine the image $C(P_m^+)$ of the set of {\it positive pure braids} $P_m^+$. 
Here, a pure braid is called positive if it admits a diagram whose crossings are all positive. In \cite[p. 308]{BGKN02}, it was conjectured that $C(P_m^+)$ is the set of matrices $M=(M_{ij})\in {\rm Mat}_m^0$ that satisfy the following conditions:
(I) $M$ is symmetric; (II) $M_{ij}$ is a nonnegative integer; and
(III) $M_{ij}=M_{jk}=0$ implies $M_{ik}=0$ for all $i,j,k$ with $1\leq i<j<k\leq m$.
It was shown in \cite{BGKN02} that this conjecture is true for $m\leq 3$. More recently, \cite{aya-yoshi, yuko-aya-yoshi} proved that the conjecture is true for $m\leq 6$. The conjecture remains open for $m\geq 7$, although there is an algorithm which, given a matrix 
$M\in {\rm Mat}_m^0$ satisfying conditions (I)–(III), either lists all positive braids whose crossing matrix is $M$ (if any exist) or declares that none exist. See \cite{GN18} for more details.
\end{rem}

\section{The extended first Johnson homomorphism of the automorphism group of a free group} \label{sec:Johnson}

In this section, we review the extended first Johnson homomorphism on the automorphism group of a free group associated with the standard Magnus expansion.
For more details, see \cite{kawazumi, morita}.
See also \cite{Kitano, Perron} for the description of the (higher) Johnson homomorphism in terms of the standard Magnus expansion.

Let $F_m$ be the free group generated by $m$ letters $x_1,\cdots, x_m$. 
The abelianization $F_m^{\rm abel} = F_m/[F_m, F_m]$ is a free abelian group of rank $m$.
For $x\in F_m$, let $[x]\in F_m^{\rm abel}$ be the image of $x$ by the canonical projection $F_m\to F_m^{\rm abel}$.
Then, the elements $[x_i]$, $1\le i \le m$, form a $\mathbb{Z}$-basis of $F_m^{\rm abel}$. We identify $F_m^{\rm abel}$ with $H$, the group introduced in Section~\ref{sec:braids}, by setting $X_i:=[x_i]$. 
For an integer $n>0$, let $H^{\otimes n}=\underbrace{H\otimes \cdots \otimes H}_{n}$ be the $n$-fold tensor product and let $H^{\otimes 0}={\mathbb Z}$.
The completed tensor algebra
\[
\widehat{T}:=\prod_{n=0}^\infty H^{\otimes n}
\]
is naturally isomorphic to the ring of non-commutative formal power series in $X_i$.

The {\it standard Magnus expansion of} $F_m$ is a unique map $\theta: F_m\to \widehat{T}$ characterized by the following properties:
\begin{enumerate}
    \item[(M0)] $\theta(e) = 1$, where $e$ is the identity element of $F_m$;
    \item[(M1)] $\theta(xy) = \theta(x)\theta(y)$ for any $x,y \in F_m$;
    \item[(M2)] $\theta(x_i)=1+X_i$ for any $1\leq i\leq m$.
\end{enumerate}
Note that $\theta(x_i^{-1}) = 1 - X_i + {X_i}^{\otimes 2} - {X_i}^{\otimes 3} + \cdots$. 
In general, for any $x\in F_m$ one can write  
\[
\theta(x)=1+\theta_1(x)+\theta_2(x)+\cdots,
\]
where $\theta_n(x) \in H^{\otimes n}$. 
In this paper, we will use only the terms of $\theta$ up to degree two. 
For the degree one part, we have
\[
\theta_1(x) = [x]
\]
for any $x\in F_m$.
For the degree two part, we will use the following property:

\begin{lem} \label{lem:theta_2}
For any $x, y\in F_m$, we have
\[
\theta_2(y^{-1}xy)=\theta_2(x)+([x]\otimes [y]-[y]\otimes [x]).
\]
\end{lem}
\begin{proof}
Modulo terms of degree $\ge 3$, we have
$\theta(y^{-1}) \equiv 1 - [y] + ( [y]^{\otimes 2} - \theta_2(y))$.
Hence
\begin{align*}
 & \theta(y^{-1}xy) = \theta(y^{-1}) \theta(x) \theta(y) \\
 \equiv &  \, \big( 1 - [y] + ( [y]^{\otimes 2} - \theta_2(y)) \big)
 \big( 1 + [x] + \theta_2(x) \big) \big( 1 + [y] + \theta_2(y) \big) \\
 \equiv & \, 1 + [x] + (\theta_2(x) + [x] \otimes [y] - [y]\otimes [x]).
\end{align*}
This proves the lemma.
\end{proof}

Let ${\rm Aut}(F_m)$ and ${\rm Aut}(H)$ be the automorphism group of the free group $F_m$ and the free abelian group $H$, respectively. 
There is a natural map ${\rm Aut}(F_m) \to {\rm Aut}(H)$ (which is known to be surjective).
For $\varphi \in {\rm Aut}(F_m)$, we denote by $|\varphi| \in {\rm Aut}(H)$ the image of $\varphi$ by this map. 
The natural action of ${\rm Aut}(H)$ on $H$ induces an action on the tensor product $H^{\otimes 2}$ and hence on the space ${\rm Hom}(H, H^{\otimes 2})$ of group homomorphisms from $H$ to $H^{\otimes 2}$.
Explicitly, this action, which we denote by $\odot$, is given as follows: for $g \in {\rm Aut}(H)$ and $F\in {\rm Hom}(H, H^{\otimes 2})$,
\[
(g\odot F)(X)=g^{\otimes 2} \big( F(g^{-1}(X)) \big), \qquad X \in H.
\]

Now let $\varphi \in {\rm Aut}(F_m)$. 
Then the map 
\[
\tau_1^{\theta}(\varphi): H \to H^{\otimes 2}, \quad 
[x] \mapsto \theta_2(x) - |\varphi|^{\otimes 2}
 \big( \theta_2(\varphi^{-1}(x)) \big)
\]
is a group homomorphism (\cite[Lemma~2.2]{kawazumi}). 
The resulting map
\[
\tau_1^{\theta} : {\rm Aut}(F_m) \to {\rm Hom}(H, H^{\otimes 2}), \quad \varphi \mapsto \tau_1^{\theta}(\varphi)
\]
is called the {\it extended first Johnson homomorphism}. 
It is actually not a homomorphism; it is rather a crossed homomorphism:

\begin{lem}[\cite{kawazumi}, Lemma 2.1] \label{lem:tau1_crossed}
For any $\varphi, \psi \in {\rm Aut}(F_m)$, 
\[
\tau_1^{\theta}(\varphi \psi) = \tau_1^{\theta}(\varphi) + |\varphi| \odot \tau_1^{\theta}(\psi).
\]
Thus we obtain the following group homomorphism:
\[
\widetilde{\tau}_1^{\theta} : {\rm Aut}(F_m) \to {\rm Hom}(H,H^{\otimes 2}) \rtimes {\rm Aut}(H), 
\quad 
\varphi \mapsto (\tau_1^{\theta}(\varphi), |\varphi|).
\]
\end{lem}

\begin{rem}
The {\it IA-automorphism subgroup of} ${\rm Aut}(F_m)$ is defined to be ${\rm IA}(F_m):= \{ \varphi \in {\rm Aut}(F_m) \mid |\varphi| = {\rm id}_H \}$.
The restriction $\tau_1 = \tau_1^{\theta}|_{{\rm IA}(F_m)}$ is a homomorphism, and it is called the {\it first Johnson homomorphism}. 
In fact, one can extend $\tau_1$ to the whole group ${\rm Aut}(F_m)$ using any {\it generalized Magnus expansion}, which is a map $\theta: F_m \to \widehat{T}$ satisfying conditions (M0), (M1), (M2) and a condition milder than (M2): 
\begin{enumerate}
\item[(M2')] $\theta(x_i) = 1 + X_i + (\text{terms of degree $\ge 2$})$ for any $1 \le i \le m$.
\end{enumerate}
See \cite{kawazumi} for more details.
\end{rem}

\section{The extended first Johnson homomorphism of a braid group} \label{sec:Candtau}
The braid group $B_m$ is isomorphic to the mapping class group of the $m$-marked disk $(D, Q_m)=(D, \{q_1,\dots ,q_m\})$ relative to the boundary (cf. \cite{artin, birman}).
For $i\in\{1,\dots, m-1\}$, the standard generator $\sigma_i$ of $B_m$ corresponds to the isotopy class of a self-homeomorphism of $(D, Q_m)$ which twists a sufficiently small disk neighborhood of the segment $\overline{q_i q_{i+1}}$ by $180^{\circ}$ counterclockwise. 
Let $q_0=(1/2, 0)\in {\partial D}$ and let $x_i$ be (the homotopy class of) a loop in $D\setminus Q_m$ with base point $q_0$ which goes along the segment $\overline{q_0q_i}$, turns around $q_i$ clockwise and comes back along the segment as depicted in Figure~\ref{fig:sigma_on_x}.
The fundamental group $\pi_1(D\setminus Q_m, q_0)$ is a free group with $m$ generators $x_1,\dots ,x_m$, and we naturally identify it with the free group $F_m$ introduced in Section~\ref{sec:Johnson}. 
The geometric background above induces a homomorphism 
\[
\Phi: B_m\to {\rm Aut}(F_m) 
\]
such that
$$\sigma_i\mapsto \left\{
\begin{array}{ccll}
\displaystyle x_i & \mapsto & x_{i+1} & \\
\displaystyle x_{i+1} & \mapsto & x_{i+1}^{-1}x_ix_{i+1} & \\
\displaystyle x_{k} & \mapsto & x_{k} & {\rm for}\ k\ne i,i+1 \\
\end{array}
\right.$$
and
$$\sigma_i^{-1}\mapsto \left\{
\begin{array}{ccll}
\displaystyle x_i & \mapsto & x_{i}x_{i+1}x_{i}^{-1} & \\
\displaystyle x_{i+1} & \mapsto & x_i & \\
\displaystyle x_{k} & \mapsto & x_{k} & {\rm for}\ k\ne i,i+1 \\
\end{array}
\right.$$
(see also Figure~\ref{fig:sigma_on_x}).
In \cite{artin}, Artin proved that $\Phi$ is injective. 
Through $\Phi$, we regard $B_m$ as a subgroup of ${\rm Aut}(F_m)$.
For any $\beta \in B_m$, the automorphism $|\Phi(\beta)| \in {\rm Aut}(H)$ is identical to the action of the permutation $|\beta| \in S_m$.

\begin{figure}    
\[
\begin{tikzpicture}[baseline=-5pt, x=4mm, y=4mm]
\draw (-5,-4) -- (5,-4) -- (5,4) -- (-5,4) -- (-5,-4); 
\draw (-1.6,0) circle[radius=0.6] node[above=5pt]{\footnotesize $i$};
\draw (1.6,0) circle[radius=0.6] node[above=5pt]{\footnotesize $i+1$};
\draw[->] (-1,0) -- (-1,-0.2);
\draw[->] (2.2,0) -- (2.2,-0.2);
\draw (0,-4) -- (-1.38,-0.51);
\draw (-1.75,-1.25) node{\small $x_i$};
\draw (0,-4) -- (1.38,-0.51);
\draw (2.25,-1.25) node{\small $x_{i+1}$};
\fill (0,-4) circle[radius=1.6pt]; 
\fill (-1.6,0) circle[radius=1.6pt];
\fill (1.6,0) circle[radius=1.6pt];
\end{tikzpicture}
\quad 
\overset{\sigma_i}{\longrightarrow}
\quad 
\begin{tikzpicture}[baseline=-5pt, x=4mm, y=4mm]
\draw (-5,-4) -- (5,-4) -- (5,4) -- (-5,4) -- (-5,-4); 
\draw (-1.6,0) circle[radius=0.6] node[above=5pt]{\footnotesize $i$};
\draw (1.6,0) circle[radius=0.6] node[above=5pt]{\footnotesize $i+1$};
\draw[->] (-1,0) -- (-1,-0.2);
\draw[->] (2.2,0) -- (2.2,-0.2);
\draw (0,-4) -- (1.38,-0.51);
\draw (-1.75,-1.25) node{\small $x_{i+1}^{-1}x_ix_{i+1}$};
\draw (2.25,-1.25) node{\small $x_{i+1}$};
\draw (-1.18,0.42) ..controls(1,3)and(3.5,2).. (4,0.25) to[bend left=15] (4,-0.5) to[bend left=35] (0,-4);
\fill (0,-4) circle[radius=1.6pt]; 
\fill (-1.6,0) circle[radius=1.6pt];
\fill (1.6,0) circle[radius=1.6pt];
\end{tikzpicture}
\]
\caption{the action of $\sigma_i$ on generators $x_1,\ldots,x_m$}
\label{fig:sigma_on_x}
\end{figure}
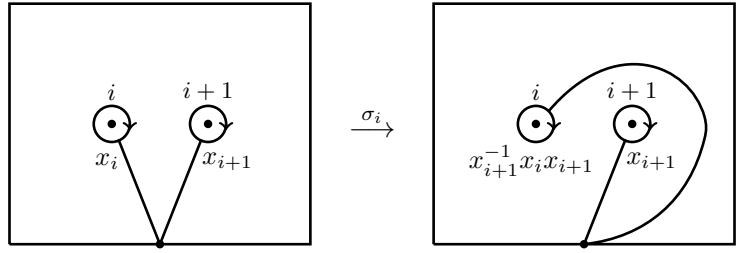

We now consider the restriction of the extended first Johnson homomorphism to $B_m$.
Let $\wedge^2 H$ be the second exterior power of $H$, which we regard as an ${\rm Aut}(H)$-submodule of $H^{\otimes 2}$ by $\wedge^2 H \to H^{\otimes 2}, X \wedge Y \mapsto X \otimes Y - Y \otimes X$. 

The following proposition was proved by Kawazumi~\cite{kawazumi06}. 
For completeness of the exposition, we give its proof as well. 

\begin{prop}[\cite{kawazumi06}, Lemma~2.1] \label{prop:tau1sigma}
For the $i$-th standard generator $\sigma_i$ of $B_m$, we have:
\begin{enumerate}
\item[$(1)$] 
$\tau_1^{\theta}(\sigma_i)(X_i)=X_i\wedge X_{i+1}$ and $\tau_1^{\theta}(\sigma_i)(X_k)=0$ if $k\ne i$; 
\item[$(2)$]
$\tau_1^{\theta}(\sigma_i^{-1})(X_{i+1})=X_i\wedge X_{i+1}$ and $\tau_1^{\theta}(\sigma_i^{-1})(X_k)=0$ if $k\ne i+1$.
\end{enumerate}
\end{prop}

\begin{proof}
We prove (1). First note that for any $k\in \{ 1,\ldots, m\}$ we have $\tau_1^{\theta}(\sigma_i)(X_k) = - |\sigma_i|^{\otimes 2}(\theta_2(\sigma_i^{-1}(x_k)))$ since $\theta_2(x_k) = 0$. Let $k \neq i, i+1$.
Then $\sigma_i^{-1}(x_k) = x_k$ and 
\[
\tau_1^{\theta}(\sigma_i)(X_k) = - |\sigma_i|^{\otimes 2}(\theta_2(\sigma_i^{-1}(x_k))) = 0.
\]
Let $k = i$. Then $\theta_2(\sigma_i^{-1}(x_i)) = \theta_2(x_ix_{i+1}x_i^{-1}) = \theta_2(x_i) + X_i \wedge X_{i+1} = X_i \wedge X_{i+1}$ by Lemma~\ref{lem:theta_2}.
Hence 
\[
\tau_1^{\theta}(\sigma_i)(X_i) = - |\sigma_i|^{\otimes 2}(\theta_2(\sigma_i^{-1}(x_i))) = - |\sigma_i|^{\otimes 2}(X_i \wedge X_{i+1})
= X_{i} \wedge X_{i+1}.
\]
Let $k = i+1$. 
Then $\sigma_i^{-1}(x_{i+1}) = x_i$ and 
\[
\tau_1^{\theta}(\sigma_i)(X_{i+1}) = - |\sigma_i|^{\otimes 2}(\theta_2(\sigma_i^{-1}(x_{i+1}))) = 0.
\]
The proof of (2) is similar. 
\end{proof}

By Proposition~\ref{prop:tau1sigma} and the fact that $\tau_1^{\theta}$ is a crossed homomorphism, we see that the restriction of $\tau_1^{\theta}$ to $B_m$ takes values in ${\rm Hom}(H, \wedge^2 H)$.

Let $\delta : {\rm Mat}_m^0 \to {\rm Hom}(H, \wedge^2 H)$ be the map defined by
\[
\delta(M)(X_i) := X_i \wedge f_i(M),
\]
where $f_i(M) \in H$ is the element defined in Section~\ref{sec:braids}. 
The map $\delta$ is injective and $S_m$-equivariant.
To see the $S_m$-equivariance, we compute 
\begin{align*}
\delta(\pi(M))(X_i) &= X_i \wedge f_i(\pi(M)) \\
&= X_i \wedge \pi (f_{\pi^{-1}(i)}(M)) \\
&= \pi^{\otimes 2}( X_{\pi^{-1}(i)} \wedge f_{\pi^{-1}(i)}(M)) \\ 
&= (\pi \odot \delta(M))(X_i).
\end{align*}
Here, in the second line we have used Lemma~\ref{lem:f_i_pi}.

\begin{thm} \label{thm:tau=C2}
\[
\tau_1^{\theta} = \delta \circ C: B_m \to {\rm Hom}(H, \wedge^2 H).
\]
In other words, for each element $\beta\in B_m$ and for each element $i\in \{1,\dots ,m\}$, we have $\tau_1^{\theta}(\beta)(X_i)=X_i\wedge f_{i}(\beta)$.
\end{thm}

\begin{proof}
Since both the maps $\tau_1^{\theta}$ and $\delta \circ C$ are crossed homomorphisms, it is sufficient to compare their values on the standard generators of $B_m$. 

Let $i,j \in \{ 1,\ldots, m\}$.
By Proposition~\ref{prop:tau1sigma}~(1), $\tau_1^{\theta}(\sigma_j)(X_i)=X_i\wedge X_{i+1}$ if $j=i$, and $\tau_1^{\theta}(\sigma_j)(X_i)=0$ otherwise. On the other hand, only the $(j+1,j)$ entry of the crossing matrix $C(\sigma_j)$ of $\sigma_j$ is $1$ and all the others are $0$. Hence, $f_i(\sigma_j)=X_{i+1}$ if $j=i$, and $f_i(\sigma_j)=0$ otherwise. Thus, $\tau_1^{\theta}(\sigma_j)(X_i)=X_i\wedge f_{i}(\sigma_j)$.
This completes the proof. 
\end{proof}

We have the following commutative diagram: 
\[
\xymatrix{
 & B_m \ar[dl]_-{\widetilde{C}} \ar[dr]^-{\widetilde{\tau}_1^{\theta}} & \\
 {\rm Mat}_m^0 \rtimes S_m \ar[rr] & & {\rm Hom}(H,\wedge^2 H) \rtimes {\rm Aut}(H).
}
\]
Here, the horizontal arrow is induced from $\delta$ and the natural inclusion $S_m \to {\rm Aut}(H)$.
Thus we can rewrite the target of the restriction of the map $\widetilde{\tau}_1^{\theta}$ to $B_m$ as follows:
\[
\widetilde{\tau}_1^{\theta} : B_m \to {\rm Hom}(H,\wedge^2 H) \rtimes S_m, \quad \beta \mapsto (\tau_1^{\theta}(\beta), |\beta|).
\]

\section{Simple braids and cords in a disk} \label{sec:Simple_braids}
In this section, we describe the diving combinatorial information of simple braids. 
Let us introduce one notation: for two elements $g$ and $h$ in a group, 
we define the conjugate of $g$ by $h$ to be $g*h= h^{-1}gh$.  

Recall that the standard generators $\sigma_i$ of $B_m$ are conjugate to each other since $\sigma_{i+1}=\sigma_i*(\sigma_{i+1}\sigma_i)$ for any $i\in \{1,\dots, m-2\}$. Any element of $B_m$ conjugate to these generators (resp.\,the inverses of these generators) is called a {\it positive} (resp.\,{\it negative}) {\it simple braid} of degree $m$. 
Let $SB_m^+$ and $SB_m^-$ the set of positive (resp.\,negative) simple braids of degree $m$. 
We also use the notation
\[
SB_m=SB_m^{+}\cup SB_m^{-}
\]
for the set of all simple braids of degree $m$. 
Recall that $|\cdot|:B_m\to S_m$ is the natural projection.
Then $|SB_m| \subset S_m$ is the set of transpositions of $S_m$.
We also remark that $|SB_m|=|SB_m^{+}|=|SB_m^{-}|$. 
For $i,j\in \{1,\ldots, m\}$ with $i<j$, let us denote by $(i\ j) \in S_m$ the transposition of $i$ and $j$, and set 
\begin{align*}
SB(i\ j)_m^{+} &=\{\beta\in SB_m^{+}\ |\ |\beta|=(i\ j)\}, \\
SB(i\ j)_m^{-} &=\{\beta\in SB_m^{-}\ |\ |\beta|=(i\ j)\}.
\end{align*} 
We also set $SB(i\ j)_m=SB(i\ j)_m^{+}\cup SB(i\ j)_m^{-}$.

Recall that $D$ is the closed domain $I\times I$ with $m$ marked points $Q_m = \{ q_1, \ldots, q_m\}$. 
For $i,j\in \{1,\dots ,m\}$ satisfying $i<j$, an $(i,j)$-{\it cord} on $(D,{Q}_m)$ is an embedding $\gamma : [0,1]\to ({\rm Int}\, D\setminus {Q}_m)\cup \{q_i,q_j\}$ such that $\gamma(0)=q_i$ and $\gamma(1)=q_j$. Two cords $\gamma_1$ and $\gamma_2$ are {\it isotopic} if they are ambiently isotopic to each other by an isotopy of $D$ that fixes $Q_m$ and $\partial D$ pointwise.

For an element $\beta$ of $SB_m^{+}$ (resp. $SB_m^{-}$), there is a cord $\gamma_{\beta}$ such that $\beta$ corresponds to the isotopy class of a positive (resp. negative) half-twist about $\gamma_{\beta}$, namely a self-homeomorphism of $(D,Q_m)$ which twists a sufficiently small
disk neighborhood of the simple arc $\gamma_{\beta}([0,1])$ by $180^{\circ}$ counterclockwise (resp. clockwise). See Figure~\ref{fig:half}. 
We call $\gamma_{\beta}$ a {\it cord of} $\beta$. A cord of a simple braid is uniquely determined up to isotopy. Conversely, the isotopy class of an $(i,j)$-cord and a sign $\varepsilon \in \{ \pm 1 \}$ uniqely determines a simple braid in $SB(i\ j)_m^{\varepsilon}$. For example, the cord of the standard generator $\sigma_i$ is the path traversing the segment $\overline{q_iq_{i+1}}$ from $q_i$ to $q_{i+1}$. 
\begin{figure}
\[
\begin{tikzpicture}[baseline=-5pt, x=3mm, y=3mm]
\draw (0,0) circle[radius=4];
\fill (2,0) circle[radius=1.6pt];
\fill (-2,0) circle[radius=1.6pt];
\draw[dotted] (2,0) -- (4,0);
\draw[dotted] (-2,0) -- (-4,0);
\draw (-2,0) -- (2,0);
\draw (0,1) node{$\gamma_{\beta}$};
\end{tikzpicture}
\quad \longrightarrow \quad
\begin{tikzpicture}[baseline=-5pt, x=3mm, y=3mm]
\draw (0,0) circle[radius=4];
\fill (2,0) circle[radius=1.6pt];
\fill (-2,0) circle[radius=1.6pt];
\draw[dotted] (2,0) to[bend left=50] (-4,0);
\draw[dotted] (-2,0) to[bend left=50] (4,0);
\draw (-2,0) -- (2,0);
\end{tikzpicture}
\]
\caption{the positive half-twist about the cord $\gamma_{\beta}$}
\label{fig:half}
\end{figure}
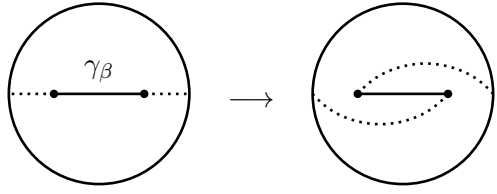

For each $k \in \{ 1,\ldots, m\}$, take a small closed disk $D_k$ centered at $q_k$, and let $l_k$ be the loop which traverses the boundary of $D_k$ clockwise.
Also, let $\alpha_k$ be the path traversing the segment $\overline{q_0q_k}$ from $q_0$ to $q_k$, where $q_0=(1/2, 0)\in {\partial D}$. The first homology group $H_1(D \setminus Q_m)$ is naturally identified with $\pi_1(D \setminus Q_m,q_0)^{\rm abel} = F_m^{\rm abel} = H$ and the homology class $[l_k]$ with $X_k$, for $1 \le k \le m$. 

Take $i, j\in \{1, \dots , m\}$ satisfying $i<j$.
Let
\[
D_{i,j}= D \setminus \bigsqcup_{k \neq i,j} {\rm Int}\, D_k.
\]
The first homology group $H_1(D_{i,j})$, which we may view as a subgroup of $H$, 
is generated by $X_k$ for $k\in \{ 1,\ldots, m\} \setminus \{i,j\}$.
For an $(i,j)$-cord $\gamma$, let $\widetilde{\gamma}$ be a loop with base point $q_0$ given by $\widetilde{\gamma}=\alpha_i\cdot \gamma \cdot \alpha_j^{-1}$. We call its first homology class
$[\widetilde{\gamma}]\in H_1(D_{i,j})$ the {\it first homology class of} $\gamma$. 

Let $\displaystyle \mathcal{H}_{i,j} = H_1(D_{i,j})=\bigoplus_{k\neq i,j} {\mathbb Z}X_k$ and let $\displaystyle \widetilde{\mathcal{H}}=\bigsqcup_{i<j} \mathcal{H}_{i,j}$. 
Let $v: SB(i\ j)_m\to \mathcal{H}_{i,j}\times \{\pm 1 \}$ be a map defined by 
$v(\beta)=([\widetilde{\gamma_{\beta}}],\varepsilon (\beta))$ for $\beta\in SB(i\ j)_m$, where $\varepsilon (\beta)=+1$ if $\beta\in SB(i\ j)_m^{+}$ or $\varepsilon (\beta)=-1$ if $\beta\in SB(i\ j)_m^{-}$. This map can be extended to a map 
\[
v:SB_m\to \widetilde{\mathcal{H}} \times \{\pm 1\}.
\]

\begin{figure}
\[
\begin{tikzpicture}[x=4mm, y=2mm]
\draw (0,14.2) node{$1$};
\draw (2,14.2) node{$2$};
\draw (4,14.2) node{$3$};
\draw (6,14.2) node{$4$};
\draw (8,14.2) node{$5$};
\draw (2,0) -- (2,6);
\draw (3,0) -- (7,6);
\draw (3,6) -- (4.8, 3.3);
\draw (5.2, 2.7) -- (7,0);
\draw (4,0) -- (4,1.1);
\draw (4,1.9) -- (4,4.1);
\draw (4,4.9) -- (4,6);
\draw (8,0) -- (8,6);
\draw (2,6) -- (2,8.7);
\draw (2,9.3) -- (2,13);
\draw (1.8,7.5) to[bend left=50] (1.8,8.7) to (2.2,9.3) to[bend right=50] (2.2,10.5);
\draw (1.8,10.6) to[bend left=20] (0,13);
\draw (4,6) -- (4,13);
\draw (3,6) -- (2.2,7.2);
\draw (7,6) -- (8.2,7.8) to[bend right=70] (8.2,9);
\draw (7.8, 9.1) to[bend left=20] (6,13);
\draw (8,6) -- (8,7.1);
\draw (8,7.8) -- (8,13);
\draw (2,0) -- (2,-2.7);
\draw (2,-3.3) -- (2,-7);
\draw (1.8,-1.5) to[bend right=50] (1.8,-2.7) to (2.2,-3.3) to[bend left=50] (2.2,-4.5);
\draw (1.8,-4.6) to[bend right=20] (0,-7);
\draw (4,0) -- (4,-7);
\draw (3,0) -- (2.2,-1.2);
\draw (7,0) -- (8.2,-1.8) to[bend left=70] (8.2,-3);
\draw (7.8, -3.1) to[bend right=20] (6,-7);
\draw (8,0) -- (8,-1.1);
\draw (8,-1.8) -- (8,-7); 
\end{tikzpicture}
\hspace{4em}
\begin{tikzpicture}[x=4mm, y=4mm]
\draw (-2,-4) -- (10,-4) -- (10,4) -- (-2,4) -- (-2,-4);  
\fill (2,0) circle[radius=1.6pt];
\fill (4,0) circle[radius=1.6pt];
\fill (8,0) circle[radius=1.6pt];
\fill (0,0) circle[radius=1.6pt];
\fill (6,0) circle[radius=1.6pt]; 
\draw (2,0) circle[radius=0.4];
\draw[->] (2.4,0) --(2.4,-0.2);
\draw (2,-1) node{\small $l_2$};
\draw (4,0) circle[radius=0.4];
\draw[->] (4.4,0) --(4.4,-0.2);
\draw (4,-1) node{\small $l_3$};
\draw (8,0) circle[radius=0.4];
\draw[->] (8.4,0) --(8.4,-0.2);
\draw (8,-1) node{\small $l_5$};
\draw (2,2.5) node{\small $\gamma_{\beta}$};
\draw(0,0) -- (0,0.5) to[bend left=30] (0.5,1) -- (2,1) to[bend left=30] (2.7,0.3) -- (2.7,-1.3) to[bend left=30] (2,-2) -- (0,-2) to[bend left=40] (-0.7,-1.3) -- (-0.7,1.3) to[bend left=40] (0,2) -- (0.5,2);
\draw[->] (0.5,2) --(1.5,2);
\draw (1.5,2) -- (2,2) to[bend left=50] (3.3,0.7) -- (3.3,-0.7) to[bend right=40] (4,-2) -- (8,-2) to[bend right=30] (9,-1) -- (9,0.5) to[bend right=30] (8.5,1) -- (6.5,1) to[bend right=30] (6,0.5) -- (6,0);
\fill (4,-4) circle[radius=1.6pt];
\draw (4,-4.7) node{\small $q_0$}; 
\draw (0,-0.7) node{\small $q_1$};
\draw (6.5,-0.7) node{\small $q_4$};
\draw[->] (4,-4) -- (3,-3);
\draw (3,-3) -- (0,0); 
\draw (2.25,-3) node{\small $\alpha_1$};
\draw[->] (4,-4) -- (4.5,-3);
\draw (4.5,-3) -- (6,0);
\draw (5.3,-3) node{\small $\alpha_4$};
\end{tikzpicture}
\]
\caption{the $5$-braid $\beta=\sigma_1 * (\sigma_2 \sigma_3^{-1} \sigma_4^{-2} \sigma_1^{-2})$ and its cord $\gamma_{\beta}$}
\label{fig:gamma_b}
\end{figure}

For example, let $\beta=\sigma_1*(\sigma_2\sigma_3^{-1}\sigma_4^{-2}\sigma_1^{-2})\in SB_5^{+}$ as in the left hand side of Figure~\ref{fig:gamma_b}. Then, a cord $\gamma_{\beta}$ of $\beta$ is shown on the right-hand side of the figure. 
We see that the first homology class $[\widetilde{\gamma_{\beta}}]$ is $2X_2-X_5$. Hence, $v(b)=(2X_2-X_5, +1) \in \mathcal{H}_{1,4}\times \{\pm 1\}$.

\begin{prop} \label{prop:v_surj}
For each $i,j\in \{1,\dots ,m\}$ with $i<j$ and for each $\varepsilon\in \{\pm 1\}$, we have $v(SB(i\ j)_m^{\varepsilon})=\mathcal{H}_{i,j}\times \{\varepsilon\}$. Therefore, the map $v: SB_m\to \widetilde{\mathcal{H}}\times \{\pm 1\}$ is surjective.
\end{prop}

\begin{proof}
We will show that any element in $\mathcal{H}_{i,j}$ is the homology class of some $(i,j)$-cord. 
Take an element $X \in \mathcal{H}_{i,j}$.
Let $l: [0,1] \to D_{i,j}$ be a loop based at $q_0$ whose homology class is $X$.
By a suitable homotopy in $(D_{i,j},q_0)$, we may assume that the following conditions are satisfied:
\begin{itemize}
    \item $l$ is an immersion with transverse double points only;
    \item $l(\frac{1}{3}) = q_i$, $l(\frac{2}{3}) = q_j$, and the curves $l([0,\frac{1}{3}])$ and $l([\frac{2}{3},1])$ are segments $\overline{q_0q_i}$ and $\overline{q_0q_j}$, respectively.
\end{itemize}
Let $\Sigma$ be the set of self-intersections of the path $l([\frac{1}{3},\frac{2}{3}])$.
If $\Sigma = \emptyset$, we are done: a reparametrization $\gamma$ of $l|_{[\frac{1}{3},\frac{2}{3}]}$ is an $(i,j)$-cord with $l = \alpha_i \cdot \gamma \cdot \alpha_j^{-1}$, and hence the homology class of $\gamma$ is $[l] = X$. 

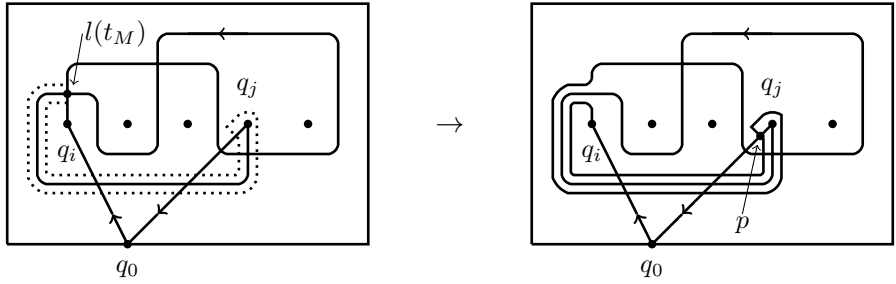
\begin{figure}
\[
\begin{tikzpicture}[baseline=-3pt, x=4mm, y=4mm]
\draw (-5,-4) -- (7,-4) -- (7,4) -- (-5,4) -- (-5,-4); 
\fill (-3,0) circle[radius=1.6pt];
\fill (-1,0) circle[radius=1.6pt];
\fill (1,0) circle[radius=1.6pt];
\fill (3,0) circle[radius=1.6pt];
\fill (5,0) circle[radius=1.6pt];
\fill (-3,1) circle[radius=1.6pt];
\draw[->, thin] (-2.5,3) node[right=-3pt]{\small $l(t_M)$} -- (-2.8,1.3);
\draw (3,0) -- (3,-1.7) to[bend left=30] (2.7,-2) -- (-3.7,-2) to[bend left=30] (-4,-1.7) -- (-4,0.7) to[bend left=30] (-3.7,1) -- (-2.3,1) to[bend left=30] (-2,0.7) -- (-2,-0.7) to[bend right=30] (-1.7,-1) -- (-0.3,-1) to[bend right=30] (0,-0.7) -- (0,2.7) to[bend left=30] (0.3,3) -- (5.7,3) to[bend left=30] (6,2.7) -- (6,-0.7) to[bend left=30] (5.7,-1) -- (2.3,-1) to[bend left=30] (2,-0.7) -- (2,1.7) to[bend right=30] (1.7,2) -- (-2.7,2) to[bend right=30] (-3,1.7) -- (-3,0);
\draw[->-] (3,3) -- (1,3);
\draw[dotted] (-3,0.7) -- (-3.6,0.7) to[bend right=20] (-3.7,0.6) -- (-3.7,-1.6) to[bend right=20] (-3.6,-1.7) -- (2.6,-1.7) to[bend right=20] (2.7,-1.6) -- (2.7,-0.5) to (2.3,-0.1) to (2.7,0.3) to[bend left=30] (3.3,0.3) to (3.3,-1.8) to[bend left=30] (2.8,-2.3) to (0.8,-2.3) -- (-3.8,-2.3) to[bend left=20] (-4.3,-1.8) -- (-4.3,0.8) to[bend left=20] (-3.8,1.3) -- (-3,1.3);
\fill (-1,-4) circle[radius=1.6pt] node[below=2pt]{\small $q_0$};
\draw (-3,-1) node{\small $q_i$};
\draw (3,1.2) node{\small $q_j$};
\draw[->] (-1,-4) -- (-1.5,-3);
\draw (-1.5,-3) -- (-3,0);
\draw (-1,-4) -- (0,-3); 
\draw[->] (3,0) -- (0,-3);
\end{tikzpicture}
\hspace{2em} \rightarrow \hspace{2em}
\begin{tikzpicture}[baseline=-3pt, x=4mm, y=4mm]
\draw (-5,-4) -- (7,-4) -- (7,4) -- (-5,4) -- (-5,-4); 
\fill (-3,0) circle[radius=1.6pt];
\fill (-1,0) circle[radius=1.6pt];
\fill (1,0) circle[radius=1.6pt];
\fill (3,0) circle[radius=1.6pt];
\fill (5,0) circle[radius=1.6pt];
\draw[->, thin] (2,-3) node[below=-3pt]{\small $p$} -- (2.5,-0.7);
\fill (2.6,-0.4) circle[radius=1.6pt];
\draw (3,0) -- (3,-1.7) to[bend left=30] (2.7,-2) -- (-3.7,-2) to[bend left=30] (-4,-1.7) -- (-4,0.7) to[bend left=30] (-3.7,1) -- (-2.3,1) to[bend left=30] (-2,0.7) -- (-2,-0.7) to[bend right=30] (-1.7,-1) -- (-0.3,-1) to[bend right=30] (0,-0.7) -- (0,2.7) to[bend left=30] (0.3,3) -- (5.7,3) to[bend left=30] (6,2.7) -- (6,-0.7) to[bend left=30] (5.7,-1) -- (2.3,-1) to[bend left=30] (2,-0.7) -- (2,1.7) to[bend right=30] (1.7,2) -- (-2.7,2) to[bend right=30] (-3,1.7) -- (-3,1.5);
\draw (-3,0) -- (-3,0.5) to[bend right=20] (-3.2,0.7);
\draw[->-] (3,3) -- (1,3);
\draw (-3.2,0.7) -- (-3.6,0.7) to[bend right=20] (-3.7,0.6) -- (-3.7,-1.6) to[bend right=20] (-3.6,-1.7) -- (2.6,-1.7) to[bend right=20] (2.7,-1.6) -- (2.7,-0.5) to (2.3,-0.1) to (2.7,0.3) to[bend left=30] (3.3,0.3) to (3.3,-1.8) to[bend left=30] (2.8,-2.3) to (0.8,-2.3) -- (-3.8,-2.3) to[bend left=20] (-4.3,-1.8) -- (-4.3,0.8) to[bend left=20] (-3.8,1.3) -- (-3.2,1.3) to[bend right=20] (-3,1.5);
\fill (-1,-4) circle[radius=1.6pt] node[below=2pt]{\small $q_0$};
\draw (-3,-1) node{\small $q_i$};
\draw (3,1.2) node{\small $q_j$};
\draw[->] (-1,-4) -- (-1.5,-3);
\draw (-1.5,-3) -- (-3,0);
\draw (-1,-4) -- (0,-3); 
\draw[->] (3,0) -- (0,-3);
\end{tikzpicture}
\]
\caption{pushing out $l(t_M)$ through $q_j$}
\label{fig:pushingout}
\end{figure}

When $\Sigma \neq \emptyset$, we modify $l$ as follows. Let $t_M$ be the maximal number in the preimage $l^{-1}(\Sigma)$.
There exists an ambient isotopy of $D_{i,j}$ fixing $q_0$ and $Q_m\setminus \{q_i, q_j\}$ pointwise which {\it pushes out} $l(t_M)$ through $q_j$, i.e. which slides the intersection $l(t_M)$ through the curve $l([t_M, \frac{2}{3}])$ towards $q_j$, passing through $q_j$ and reaching onto the interior of $\overline{q_0q_j}$. 
See Figure~\ref{fig:pushingout}. Here, $p$ stands for the result of pushing out $l(t_M)$ by the ambient isotopy. (See also \cite{kamada-matsumoto}.)
This isotopy reduces the number of points in $\Sigma$.
By repeating this process, we may assume that $\Sigma = \emptyset$. This completes the proof.
\end{proof}

For $k\in \{1,\dots, m\}\setminus \{i,j\}$, let $\zeta_k$ (resp. $\eta_k$) be the intersection of the vertical segment $\{k/(m+1)\} \times [1/2,1]$ (resp. $\{ k/(m+1) \} \times [0,1/2]$) with $D_{i,j}$, oriented upwards.  
We use the same letters $\zeta_k$ and $\eta_k$ for their homology classes in $H_1(D_{i,j}, \partial D_{i,j})$. 
Note that we have
\[
\zeta_k + \eta_k = 0 \in H_1(D_{i,j},\partial D_{i,j}).
\]
Using the intersection pairing 
$
(\ \cdot \ ) : H_1(D_{i,j}) \times H_1(D_{i,j}, \partial D_{i,j}) 
\to \mathbb{Z}, 
$
we can write 
\[
[\widetilde{\gamma_{\beta}}] = \sum_{k \neq i,j} ([\widetilde{\gamma_{\beta}}] \cdot \zeta_k) X_k.
\]

Since the elements $f_i(\beta)$ recover the crossing matrix $C(\beta)$, the following theorem gives a formula for the crossing matrix of a simple braid. 

\begin{thm} \label{prop:f_i_for_simple}
Let $\beta \in SB(i\ j)_m$.
\begin{enumerate}
\item[$(1)$]
If $\beta$ is positive, then $f_i(\beta) = [\widetilde{\gamma_{\beta}}] + X_j$ and $f_j(\beta) = - [\widetilde{\gamma_{\beta}}]$.

\item[$(2)$]
If $\beta$ is negative, then 
$f_i(\beta) = [\widetilde{\gamma_{\beta}}]$ and $f_j(\beta) = - [\widetilde{\gamma_{\beta}}] - X_i$.

\item[$(3)$]
For $k \neq i,j$, we have
\[
f_k(\beta) = \begin{cases} (([\widetilde{\gamma_{\beta}}]\cdot \zeta_k - 1)) (X_i - X_j) & \text{if $i< k < j$}; \\
([\widetilde{\gamma_{\beta}}]\cdot \zeta_k) (X_i - X_j) & \text{if $k<i$ or $j<k$}.
\end{cases}
\]
\end{enumerate}
\end{thm}

\begin{proof}
For $k\in \{1,\dots, m\}$, $D_k$ has been defined as a small disk centered at $q_k$. We may assume that all these disks have the same radius. 
For simplicity, we denote $\gamma = \gamma_{\beta}$.
By a suitable isotopy, we can arrange that the image of $\gamma$ lies in $D_{i,j}$ and that there exists a real number $\varepsilon > 0$ such that
\begin{itemize}
\item $\gamma(t) = q_i + t \overrightarrow{q_iq_j}$ for $t \in [0,\varepsilon] \cup [1-\varepsilon,1]$, 
\item $\gamma(t) \notin D_i \cup D_j$ for $t \in [\varepsilon, 1-\varepsilon]$, and
\item 
on the interval $[\varepsilon, 1-\varepsilon]$, $\gamma$ is transverse to $\delta_k$ and $\eta_k$ for all $k\in \{ 1,\ldots,m\} \setminus \{i,j\}$, as well as to the vertical segments $\delta^-_l:=\{ l/(m+1)\} \times [0,1/2)$ and $\delta^+_l:=\{l/(m+1)\} \times (1/2,1]$ for $l \in \{ i,j\}$, both oriented upwards. 
\end{itemize}
See Figure~\ref{fig:f_i_simple}.
Using $\gamma$, we take a generic geometric braid representing $\beta$ with the following properties.
First, for any $k \neq i,j$, the $k$-th string is stationary, namely it is given by $t \mapsto (q_k, t) \in D \times I$.
Next, the $i$-th and $j$-th strings are given by the following paths.
Here, we divide the unit interval $I$ into three parts: $I = I_1 \cup I_2 \cup I_3$, where $I_1=[0,(1-\varepsilon)/2]$, $I_2=[(1-\varepsilon)/2,(1+\varepsilon)/2]$ and $I_3=[(1+{\varepsilon})/2, 1]$.

\begin{enumerate}
\item[(i)]
On $I_1$, the $i$-th string approaches a point close to $q_j$ traversing $\gamma$, while the $j$-th string is stationary at $q_j$; more concretely, the $i$-th string is given by $t \mapsto (\gamma(2t),t)$ and the $j$-th string by $t \mapsto (q_j,t)$  for $t\in I_1$.

\item[(ii)]
On $I_2$, the $i$-th and $j$-th strings exchange their positions by a rotation along the circle whose diameter is the line segment between $\gamma(1-\varepsilon)$ and $q_j$, with constant angular velocity $\pi/\varepsilon$ (resp. $-\pi/\varepsilon$) if $\beta$ is positive (resp. negative);
more concretely, if $\beta$ is positive, the $i$-th string is given by $t\mapsto (\omega_i(t), t)$ and the $j$-th string by $t\mapsto (\omega_j(t), t)$ for $t\in I_2$, where $\omega_i(t)$ and $\omega_j(t)$ are the points on the circle above with arguments $\pi (t-(1-\varepsilon)/2)$ and $\pi(t+(1+\varepsilon)/2)$, respectively. 

\item[(iii)]
On $I_3$, the $i$-th string is stationary at $q_j$, while the $j$-th string approaches $q_i$ traversing the inverse path of $\gamma$; more concretely, the $i$-th string is given by $t \mapsto (q_j,t)$ and the $j$-th string by $t \mapsto (\gamma(2-2t),t)$ for $t\in I_3$.
\end{enumerate}

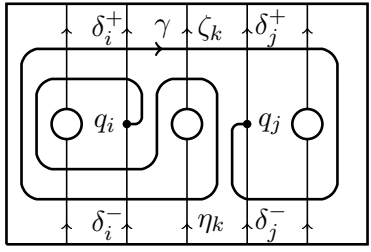
\begin{figure}
\[
\begin{tikzpicture}[x=4mm, y=4mm]
\draw (-4,-4) -- (8,-4) -- (8,4) -- (-4,4) -- (-4,-4);
\fill (0,0) circle[radius=1.6pt] node[left]{$q_i$};
\fill (4,0) circle[radius=1.6pt] node[right]{$q_j$};
\draw (-2,0) circle[radius=0.5];
\draw (2,0) circle[radius=0.5];
\draw (6,0) circle[radius=0.5];
\draw (0,0) -- (0.3,0) to[bend right=30] (0.5,0.2) -- (0.5,1) to[bend right=30] (0,1.5) -- (-2.5,1.5) to[bend right=30] (-3,1) -- (-3,-1) to[bend right=30] (-2.5,-1.5) -- (0.5,-1.5) to[bend right=30] (1,-1) -- (1,1) to[bend left=30] (1.5,1.5) -- (2.5,1.5) to[bend left=30] (3,1) -- (3,-2) to[bend left=30] (2.5,-2.5) -- (-3,-2.5) to[bend left=30] (-3.5,-2) -- (-3.5,2) to[bend left=30] (-3,2.5) -- (6.5,2.5) to[bend left=30] (7,2) -- (7,-2) to[bend left=30] (6.5,-2.5) -- (4,-2.5) to[bend left=30] (3.5,-2) -- (3.5,-0.2) to[bend left=30] (3.7,0) -- (4,0);
\draw[->] (1,2.5) -- (1.2,2.5) node[above]{$\gamma$}; 
\draw[semithick] (-2,0.5) -- (-2,4);
\draw[->, semithick] (-2,3) -- (-2,3.2);
\draw[semithick] (-2,-4) -- (-2,-0.5);
\draw[->, semithick] (-2,-3.5) -- (-2,-3.3);
\draw[semithick] (0,-4) -- (0,4);
\draw[->, semithick] (0,-3.5) -- (0,-3.3) node[left=-3pt]{$\delta^-_i$};
\draw[->, semithick] (0,3) -- (0,3.2) node[left=-3pt]{$\delta^+_i$};
\draw[semithick] (2,0.5) -- (2,4);
\draw[->, semithick] (2,3) -- (2,3.2) node[right]{$\zeta_k$};
\draw[semithick] (2,-4) -- (2,-0.5);
\draw[->, semithick] (2,-3.5) -- (2,-3.3) node[right]{$\eta_k$};
\draw[semithick] (4,-4) -- (4,4);
\draw[->, semithick] (4,-3.5) -- (4,-3.3) node[right=-1pt]{$\delta^-_j$};
\draw[->, semithick] (4,3) -- (4,3.2) node[right=-1pt]{$\delta^+_j$};
\draw[semithick] (6,0.5) -- (6,4);
\draw[->, semithick] (6,3) -- (6,3.2);
\draw[semithick] (6,-4) -- (6,-0.5);
\draw[->, semithick] (6,-3.5) -- (6,-3.3);
\end{tikzpicture}
\]
\caption{the arrangement of the cord $\gamma$}
\label{fig:f_i_simple}
\end{figure}

(1) Suppose that $\beta$ is positive.
First we compute the coefficients of $X_k$ in $f_i(\beta)$ and $f_j(\beta)$ for $k\neq i,j$.
For $f_i(\beta)$, the coefficient of $X_k$ comes from the motion of the $i$-th string on the interval $I_1$.
It counts with sign all intersections of $\gamma$ and $\zeta_k$.
Since the paths $\alpha_i$ and $\alpha_j$ are disjoint from $\zeta_k$, we conclude that the coefficient of $X_k$ in $f_i(\beta)$ is equal to $([\widetilde{\gamma}]\cdot \zeta_k)$. 
As for $f_j(\beta)$, the coefficient of $X_k$ comes from the motion of the $j$-th string on $I_3$, and it is negative to that of $f_i(\beta)$.

Next we compute the coefficient of $X_j$ in $f_i(\beta)$.
We count crossings of the $i$-th string over the $j$-th string on the three parts $I_1$, $I_2$ and $I_3$.
On $I_1$, such crossings correspond to the intersections of $\gamma$ with $\delta^+_j$, each contributing 
an amount equal to the sign of the intersection.
On $I_2$, we obtain a positive crossing.
On $I_3$, such crossings correspond to the intersections of $\gamma$ with $\delta^-_j$, each contributing an amount equal to the negative of the sign of the intersection.
By our arrangement of $\gamma$, the algebraic intersection number of $\gamma$ with the union of $\delta_j^+$ and $\delta_j^-$ is zero.
Hence the contributions from $I_1$ and $I_3$ cancel, and the coefficient of $X_j$ in $f_i(\beta)$ is one.

In a similar manner, we see that the coefficient of $X_i$ in $f_j(\beta)$ is zero.
In this case, the contributions from $I_1$ and $I_3$ cancel as well, and 
we have no contribution from $I_2$.
This completes the proof of (1).

The case (2) is proved similarly, so we omit it.

(3) Similarly to the case (1), we have
\[
f_k(\beta) = - (\gamma \cdot \eta_k) (X_i - X_j),
\]
where $(\gamma \cdot \eta_k)$ is the number of positive intersections of $\gamma$ and $\eta_k$ minus the number of negative intersections of $\gamma$ and $\eta_k$.
Furthermore, since $\widetilde{\gamma} = \alpha_i \cdot \gamma \cdot \alpha_j^{-1}$, we have 
\[
(\gamma \cdot \eta_k) = 
\begin{cases} 
([\widetilde{\gamma}]\cdot \eta_k) + 1 & \text{if $i < k < j$} \\
([\widetilde{\gamma}]\cdot \eta_k) & \text{if $k < i$ or $j < k$.}
\end{cases}
\]
Since $([\widetilde{\gamma}]\cdot \eta_k) = - ([\widetilde{\gamma}]\cdot \zeta_k)$, we obtain the desired formula.
\end{proof}

\begin{cor} \label{cor:homology-Jhonson}
Fix $i,j\in \{ 1,\ldots,m\}$ with $i<j$ and $\varepsilon \in \{\pm 1\}$. For $\beta, \beta'\in SB(i\ j)_m$, we have $[\widetilde{\gamma_{\beta}}] =[\widetilde{\gamma_{\beta'}}]$ if and only if ${\tau}_1^{\theta}(\beta)={\tau}_1^{\theta}(\beta')$.
\end{cor}
\begin{proof}
This follows from Theorem~\ref{thm:tau=C2} and Theorem~\ref{prop:f_i_for_simple}.
\end{proof}

Recall the group homomorphism $\widetilde{\tau}_1^{\theta}: B_m\to {\rm Hom}(H,\wedge^{2}H)\rtimes S_m$ defined by $\widetilde{\tau}_1^{\theta}(\beta)=(\tau_1^{\theta}(\beta), |\beta|)$. For each $\varepsilon\in \{\pm 1\}$ and each element $(f, \varepsilon)\in  \mathcal{H}_{i,j} \times \{\varepsilon \}$, there exists an element $\beta\in SB_m^{\varepsilon}$ such that $(f, \varepsilon)=v(\beta) = ([\widetilde{\gamma_\beta}],\varepsilon)$ by Proposition~\ref{prop:v_surj}.
By Corollary~\ref{cor:homology-Jhonson}, the value $\widetilde{\tau}_1^{\theta}(\beta)$ does not depend on the choice of $\beta$.
Thus, the map
\[ 
\mu: \mathcal{H}_{i,j} \times \{\varepsilon\}\to \widetilde{\tau}_1^{\theta}(SB(i\ j)_m^{\varepsilon}), \quad 
\mu((f,\varepsilon))=\widetilde{\tau}_1^{\theta}(\beta), 
\]
where $\beta$ is an element of $SB_m^{\varepsilon}$ such that $(f,\varepsilon)=v(\beta)
$, is well-defined.
This map is extended to a map $\mu: \widetilde{\mathcal{H}}\times \{\pm 1\} \to \widetilde{\tau}_1^{\theta}(SB_m)$. 
By construction, the following diagram commutes: 
\[
\xymatrix{
 & SB_m \ar[dl]_-{v} \ar[dr]^-{\widetilde{\tau}_1^{\theta}} & \\
 \widetilde{\mathcal{H}} \times \{ \pm 1 \} \ar[rr]_{\mu} & & \widetilde{\tau}_1^{\theta}(SB_m).
}
\]
By Corollary~\ref{cor:homology-Jhonson}, we have the following proposition.

\begin{prop} \label{prop:mu_is_bij}
The map $\mu: \mathcal{H}_{i,j}\times \{\varepsilon\}\to \widetilde{\tau}_1^{\theta}(SB(i\ j)_m^{\varepsilon})$ is bijective. Therefore, the map $\mu: \widetilde{\mathcal{H}}\times \{\pm 1\} \to \widetilde{\tau}_1^{\theta}(SB_m)$ is bijective.
\end{prop}

In a future work, we plan to discuss an application of this proposition to the Hurwitz action of braid groups. 

\section*{Acknowledgments}

This research is supported by JSPS KAKENHI Grant Number 19K03508, 23K03121 and 24K00520.


\begin{thebibliography}{99}
\bibitem{artin}E. Artin: 
\textit{Theorie der Z\"{o}pfe},  
Abh. Math. Sem. Univ. Hamburg {\bf 4} (1925), 47--72.

\bibitem{birman}J. Birman:
\textit{Braids, Links and Mapping Class Groups}, 
Ann. of Math. 
Stud. {\bf 82}, Princeton Univ. Press, 1974.

\bibitem{BGKN02}J. Burillo, M. Gutierrez, S. Krsti\'{c} and Z. Nitecki:
\textit{Crossing matrices and Thurston's normal form for braids},
Topology Appl. {\bf 118} (2002), no.\,3, 293--308.

\bibitem{Dehornoy}P. Dehornoy, 
\textit{Extending the Hurwitz action to shelves that are not racks}, Kyoto University RIMS K{\^o}ky{\^u}roku, vol. 1960, pp. 73--85, 2015. 

\bibitem{GN18}M. Gutierrez and Z. Nitecki:
\textit{Crossing Matrices of Positive Braids},
preprint, arXiv:1805.12189 (2018).

\bibitem{kamada}S. Kamada: 
\textit{Braid and Knot Theory in Dimension
Four, Math. Surveys and Monographs} {\bf 95},
Amer. Math. Soc., 2002. 

\bibitem{kamada-matsumoto}S. Kamada and Y. Matsumoto:
\textit{Word representation of cords on a punctured plane}
, Topology Appl. {\bf 146-147} (2005), 21--50.

\bibitem{kawazumi}N. Kawazumi: 
\textit{Cohomological aspects of Magnus expansions}, preprint, math.GT/0505497 (2005). 

\bibitem{kawazumi06}N. Kawazumi:
\textit{Twisted Morita-Mumford classes on braid groups}, in: 
Groups, homotopy and configuration spaces, Geom. Topol. Monogr., {\bf 13}, 293--306 (2008).

\bibitem{Kitano} T. Kitano:
\textit{Johnson's homomorphisms of subgroups of the mapping class group, the Magnus expansion and Massey higher products of mapping tori}, Topology Appl. {\bf 69} (1996), no. 2, 165--172.

\bibitem{morita}S. Morita: 
\textit{The extension of Johnson's homomorphism from the Torelli group to the mapping class group}, Invent. math. {\bf 111} (1993), no. 1, 197--224.

\bibitem{Perron}B. Perron: 
\textit{Homomorphic extensions of Johnson homomorphisms via Fox calculus},
Ann. Inst. Fourier (Grenoble) {\bf 54} (2004), no. 4, 1073--1106.

\bibitem{aya-yoshi}A. Shimizu and Y. Yaguchi:
\textit{Characterization of the OU matrix of a braid diagram}, Topology Appl. {\bf 373} (2025), 109440, 18 pp.

\bibitem{yuko-aya-yoshi}Y. Ozawa, A. Shimizu and Y. Yaguchi:
\textit{The CN matrix of a pure braid projection}, to appear in J. Knot Theory Ramifications. (arXiv:2506.08659)

\bibitem{thurston}W. Thurston: Braid groups, in: D. Epsteinetal. (Eds.), Word Processing in
 Groups, Jonesand Bartlett, Boston, MA, 1992, pp.181–209.
\end{thebibliography}
\end{document}